\documentclass[leqno,onetabnum]{siamltex704}

\usepackage{amsmath} 
\usepackage{amssymb} 
\usepackage[dvips]{epsfig}
\usepackage{verbatim} 
\usepackage{psfrag}

\def\RR{\ensuremath{\mathbb R}}
\newcommand{\Rn}{{\RR}^n}
\newcommand{\Rnn}{\RR^{n\times n}}
\newcommand{\spanl}{\mathrm{span}}

\def\fgax{\ensuremath{\mathcal F}_\gamma(x)}

\def\bgax{\ensuremath{\mathcal B}_\gamma(x)}

\newcommand{\lmu}{{\mathcal L}(\mu)}
\newcommand{\lmun}{{\mathcal L}(\mu_0)}

\begin{document}

\pagestyle{myheadings}
\thispagestyle{plain}
\markboth{K. Neymeyr}
{Preconditioned steepest descent}

\title{A geometric convergence theory for the \\preconditioned steepest
descent iteration}

\author{Klaus Neymeyr\thanks{
         Universit\"at Rostock, Institut f\"ur Mathematik, 
         Ulmenstra{\ss}e 69, 18057 Rostock, Germany.
	  }}

\vspace{-1.2in}

\vspace{.9in}
\setcounter{page}{1}
\maketitle

\begin{abstract}
Preconditioned gradient iterations for very large eigenvalue problems 
are efficient solvers with growing popularity.
However, only for the simplest preconditioned eigensolver, namely the preconditioned 
gradient iteration (or preconditioned inverse iteration) with fixed step size, sharp non-asymptotic
convergence estimates are known and these estimates require an ideally scaled preconditioner.
In this paper a new sharp convergence estimate is derived for the {\em preconditioned 
steepest descent iteration} which combines the preconditioned gradient iteration with
the Rayleigh-Ritz procedure for optimal line search convergence acceleration.
The new estimate always improves that of the fixed step size iteration.
The practical importance of this new estimate is that arbitrarily scaled preconditioners can be used. 
The Rayleigh-Ritz procedure implicitly computes the optimal scaling. 
\end{abstract}


\section{Introduction}
The topic of this paper is a convergence analysis of a preconditioned gradient
iteration with optimal step-length scaling in order to compute the 
smallest eigenvalue of the generalized eigenvalue problem 
\begin{align}
   \label{e.mgevp}
   Ax_i=\lambda_i Bx_i
\end{align}
for symmetric positive definite matrices $A,B\in\Rnn$.
A typical source of (\ref{e.mgevp}) is an eigenproblem
for a self-adjoint and elliptic partial differential operator whose
weak form reads
\begin{equation}
   \label{e.weak}
   a(u,v) = \lambda \,(u,v),  \qquad \forall v \in H(\Omega).
\end{equation}
The bilinear form $a(\cdot, \cdot)$ is associated with the partial
differential operator and an $L^2(\Omega)$ inner product 
$(\cdot, \cdot)$ appears on the right side. 
Further $u$ is an eigenfunction and $\lambda$ an eigenvalue  if (\ref{e.weak})
is satisfied for all $v$ in an appropriate Hilbert space $H(\Omega)$. 
A finite element discretization of (\ref{e.weak})
results in (\ref{e.mgevp}). Then $A$ is called the discretization matrix
and $B$ the mass matrix. These matrices are typically sparse and very large.

The eigenvalues of (\ref{e.mgevp}) are enumerated in increasing order
$0<\lambda_1\leq\lambda_2\leq\ldots\leq\lambda_n$.
The smallest eigenvalue $\lambda_1$ and an associated eigenvector 
can be computed by means of an iterative minimization
of the Rayleigh quotient
\begin{align}
  \label{e.rayquo}
  \rho(x)=\frac{(x,Ax)}{(x,Bx)},
\end{align} 
where  $(\cdot,\cdot)$ denotes the Euclidean inner product.
To this end the simplest preconditioned gradient iteration corrects
a current iterate $x$ in the
direction of the negative preconditioned gradient of the Rayleigh quotient
to form the next iterate $x'$
\begin{equation}
  \label{e.preeig}
   x'= x-T(Ax-\rho(x) Bx).
\end{equation}
Therein $T$ is a symmetric positive definite matrix and is called 
the preconditioner.
This fixed-step-length preconditioned iteration is analyzed in
\cite{TEV2000,KNN2003,KNY1998,KNN2009}; see also the references in \cite{BKP1996}.

Appropriate preconditioners $T$ are available in various ways; especially 
for the operator eigenproblem (\ref{e.weak})
multi-grid or multi-level preconditioners are available.
In this context the quality of the preconditioner is typically controlled 
in terms of a real parameter $\gamma\in[0,1)$ in a way that
  \begin{equation}
    \label{e.bass}
    (1-\gamma)(z,T^{-1}z)\leq (z,Az) \leq (1+\gamma)(z,T^{-1}z),
                                   \quad\forall z\in\Rn, 
  \end{equation}
or equivalently, that the spectral radius of the error propagation matrix
$I-TA$ is bounded by $\gamma$.
  
The following result for the convergence of (\ref{e.preeig})
is known from \cite{KNN2003,KNN2009}; the convergence analysis interprets
this preconditioned iteration as a preconditioned inverse iteration and
makes use of the underlying geometry.
\begin{theorem}
  \label{t.pinres}
If $\lambda_i\leq\rho(x)<\lambda_{i+1}$ then for $x'$ given by (\ref{e.preeig})
and assuming (\ref{e.bass}) it holds that $\rho(x')\leq\rho(x)$ and 
either $\rho(x')\leq\lambda_i$ or
  \begin{equation}
    \label{e.rhoest}
    \frac{\rho(x')-\lambda_i}{\lambda_{i+1}-\rho(x')}\leq
    \sigma^2 
    \frac{\rho(x)-\lambda_i}{\lambda_{i+1}-\rho(x)}, \qquad   
    \sigma= \gamma +(1-\gamma)\frac{\lambda_i}{\lambda_{i+1}}.
  \end{equation}
\end{theorem}

Thm.~\ref{t.pinres} is up to now the only known sharp estimate for 
this and various improved and faster converging preconditioned gradient type eigensolvers.
The most popular of these improved solvers are the
preconditioned steepest descent iteration (PSD) and the locally optimal preconditioned
conjugate gradients (LOPCG) iteration (and also their block variants)
\cite{KNY1998}.
All these eigensolvers apply the
Rayleigh-Ritz procedure to proper subspaces of iterates for convergence acceleration, see
\cite{KNN2003etna}. A systematic hierarchy of these preconditioned gradient iterations and
their variants for exact inverse preconditioning (which amounts to certain Invert-Lanczos
processes \cite{NEY2009}) has been suggested in \cite{NEY2001H}.
The aim of this paper is to prove a new sharp convergence estimate for the 
preconditioned steepest descent iteration (PSD).

\subsection{Assumptions on the preconditioner}
\label{ss.genass}
A drawback of Thm.~\ref{t.pinres} is its assumption (\ref{e.bass}) on the preconditioner $T$.
The existence of constants $1\pm\gamma$ with $\gamma<1$ is not guaranteed for arbitrary
(multigrid) preconditioners, but can always be ensured after a proper scaling of the preconditioner.
To make this clear, take an   
arbitrary pair of symmetric positive definite matrices $A,T\in\Rnn$. Then 
constants  $\gamma_1,\,\gamma_2>0$ exist, so that 
the spectral equivalence
 \begin{equation}
    \label{e.bass12}
    \gamma_1(z,T^{-1}z)\leq (z,Az) \leq \gamma_2 (z,T^{-1}z),
                                   \quad\forall z\in\Rn
  \end{equation}
holds.
If a preconditioner $T$ satisfies (\ref{e.bass12}), then the scaled preconditioner
$(2/(\gamma_1+\gamma_2))T$ fulfills  (\ref{e.bass}) with
\begin{align}
  \label{e.ga12} 
  \gamma=\frac{\gamma_2-\gamma_1}{\gamma_1+\gamma_2}.
\end{align}
A clear benefit of the preconditioned steepest descent iteration is, that
by computing the optimal step length parameter $\vartheta_\mathrm{opt}$, see Eq.~(\ref{e.psd}),
the scaling parameter $2/(\gamma_1+\gamma_2)$ is determined implicitly.
Therefore, we can use the assumption (\ref{e.bass12}) or alternatively the more 
convenient form (\ref{e.bass}). This guarantees the practical applicability of
the preconditioned steepest descent iteration for any preconditioner satisfying (\ref{e.bass12})
or in its scaled form satisfying (\ref{e.bass}).

\subsection{The optimal-step-length iteration: Preconditioned steepest descent}

A disadvantage of the gradient iteration (\ref{e.preeig}) is its fixed step length 
resulting in a non-optimal new iterate $x'$. An obvious improvement is to compute $x'$
as the minimizer of the Rayleigh quotient (\ref{e.rayquo}) in the affine space
$\{x-\vartheta T(Ax-\rho(x) Bx);\,\vartheta\in\RR\}$. That means we consider the 
optimally scaled iteration 
\begin{equation}
   \label{e.psd} 
      x' =x-\vartheta_\mathrm{opt} T(Ax-\rho(x) Bx)
\end{equation}
with the optimal step length
$$ \vartheta_\mathrm{opt}=\arg\min_{\vartheta\in\RR}
       \rho(x-\vartheta T(Ax-\rho(x) Bx) )$$
is considered.
This iteration is called the {\em preconditioned steepest
descent iteration} (PSD), \cite{TEV2000,KNN2003etna,OVT2006}. 
Computationally one gets $x'$ and its Rayleigh
quotient $\rho(x')$ by the Rayleigh-Ritz procedure.
If $T(Ax-\rho(x) Bx)$ is not an eigenvector then
$(x',\rho(x'))$ is a Ritz pair of $(A,B)$ with respect 
to the column space of $[x,T(Ax-\rho(x) Bx)]$. 
As (\ref{e.psd}) aims at a minimization of the Rayleigh
quotient, $\rho(x')$ is the smaller Ritz value and $x'$ is
an associated Ritz vector.
The Rayleigh-Ritz procedure computes the optimal step length implicitly;
the step length is determined by the components of the associated 
eigenvector of Rayleigh-Ritz projection matrices.
Consequently the preconditioned steepest descent iteration converges faster 
than the fixed-step-length scheme (\ref{e.preeig}) since
\begin{align}
  \label{e.psdpin} 
   \rho(x-\vartheta_\mathrm{opt} T(Ax-\rho(x) Bx))\leq \rho(x-T(Ax-\rho(x) Bx)).
\end{align}
Therefore Thm.~\ref{t.pinres} serves as a trivial upper estimate for 
the accelerated iteration (\ref{e.psd}).  The aim of this paper is to prove 
the following sharp convergence estimate for (\ref{e.psd}).

\begin{theorem}
  \label{t.psd}
Let $x\in\Rn$ and $x'$ be the PSD iterate given by (\ref{e.psd}). The  preconditioner
$T$ is assumed to satisfy (\ref{e.bass12}).  
If $\lambda_i\leq\rho(x)< \lambda_{i+1}$, $i=1,\ldots,n-1$, then $\rho(x')\leq\rho(x)$ and 
either $\rho(x')\leq\lambda_i$ or
\begin{equation}
 \label{e.psdest}
  \begin{aligned}
    &\frac{\rho(x')-\lambda_i}{\lambda_{i+1}-\rho(x')}\leq
    \sigma^2 
    \frac{\rho(x)-\lambda_i}{\lambda_{i+1}-\rho(x)}, \\[0.4em]
  &\text{with } \quad  \sigma=\frac{\kappa+\gamma(2-\kappa)}{(2-\kappa)+\gamma\kappa}, 
  \qquad 
      \kappa=\frac{\lambda_i(\lambda_n-\lambda_{i+1})}{\lambda_{i+1}(\lambda_n-\lambda_i)}
  \end{aligned}
\end{equation}
and $\gamma:=(\gamma_2-\gamma_1)/(\gamma_1+\gamma_2)$. If $\gamma_1=1-\gamma$ and $\gamma_2=1+\gamma$
as in (\ref{e.bass}), then $(\gamma_2-\gamma_2)/(\gamma_1+\gamma_2)=\gamma$.
The estimate is sharp and can be attained for $\rho(x)\to\lambda_i$ in the 3D invariant subspace 
associated with the eigenvalues $\lambda_i$, $\lambda_{i+1}$ and $\lambda_n$, $i+1\neq n$.
\end{theorem}

The limit case $\gamma=0$ of Thm.~\ref{t.psd} is an estimate for the convergence
of the steepest descent iteration which minimizes the Rayleigh
quotient in the space $\spanl\{x,A^{-1}Bx\}$. Then the convergence estimate 
(\ref{e.psdest}) reads
$$  \frac{\rho(x')-\lambda_i}{\lambda_{i+1}-\rho(x')}\leq
    \left(\frac{\kappa}{2-\kappa}\right)^2
    \frac{\rho(x)-\lambda_i}{\lambda_{i+1}-\rho(x)} $$
with $\kappa$ given by (\ref{e.psdest}).
A proof of this result (in the general setup of steepest ascent and steepest descent
for $A$ and $A^{-1}$) has recently been given in \cite{NOZ2011}; for the smallest eigenvalue
(that is for $i=1$) the estimate was proved in \cite{KNS1991}.
This paper generalizes this result
on steepest decent for $A^{-1}M$ to the preconditioned variant of this iteration. For the following analysis
we always assume a properly scaled preconditioner satisfying (\ref{e.bass}). If $T$ fulfills
(\ref{e.bass2}) we use $(2/(\gamma_1+\gamma_2))T$ 
(and call the scaled preconditioner once again $T$) so that $\gamma$ is given by (\ref{e.ga12})
and (\ref{e.bass}) is fulfilled. This substitution does not restrict the generality of the approach
since the scaling constant is implicitly computed with $\vartheta_{\mathrm opt}$ in the 
Rayleigh-Ritz procedure. We prefer to work with (\ref{e.bass}) since this allows to set up the proper geometry
for the following proof.

Only few convergence estimates for PSD have been published. Of major importance are
the  work of Samokish \cite{SAM1958}, the results of Knyazev given in
Thm.~3.3 together with Eq.~(3.3) in
\cite{KNY1987} and further the results of Ovtchinnikov \cite{OVT2006}. 
Knyazev uses similar assumptions
and applies Chebyshev polynomials to derive the convergence estimate. 
Ovtchinnikov in \cite{OVT2006} derives an asymptotic convergence factor
which represents the average error reduction per iteration; further
non-asymptotic  estimates are proved under specific assumptions on the preconditioner.
The result of Samokish (only available in Russian) is reproduced
in a finite-dimensional non-asymptotic form as Thm.~2.1 in \cite{OVT2006}; see also Cor. 6.4 and the
following paragraph in \cite{OVT2006} for a critical discussion and comparison of these estimates.
Due to different assumptions and a different form of the convergence estimates these results are
not easy to compare with (\ref{e.psdest}); an important difference is that in Thm.~\ref{t.psd} the restrictive
assumption $\rho(x)<\lambda_2$ is not needed.

\subsection{Overview}

This paper is organized as follows. In Sec.~2 the geometry of PSD is introduced.
Further the problem is reformulated in terms of reciprocals of the eigenvalues
which makes the geometry of PSD accessible within the Euclidean space.
Sec.~3 gives a proof that PSD attains its poorest convergence in a three-dimensional
invariant subspace of the $\Rn$. Sec.~4 contains a mini-dimensional analysis of PSD.
Finally the three-dimensional convergence estimates are embedded into the full $\Rn$ which
completes the convergence analysis.

\section{The geometry of the preconditioned steepest descent iteration}

For the analysis of the preconditioned steepest descent iteration it is 
convenient to work with the linear pencil $B-\mu A$ (instead of $A-\lambda B$).
The advantage is that the $A$-norm by a proper basis transformation turns into
the Euclidean norm, see below.
A further benefit of this representation is that a generalization to a
symmetric positive semidefinite or even only a symmetric $B$ is possible
(cf.~the analysis of (\ref{e.preeig}) in \cite{KNN2009}).
Hence for the pencil $B-\mu A$ the eigenvalues $\mu_i$ are given by
$$ Bx_i=\mu_iAx_i \quad \text{ with } \mu_i=1/\lambda_i, \quad i=1,\ldots,n.$$
Therefore the problem is to compute the largest eigenvalue $\mu_1$ by maximizing
the inverse of the Rayleigh quotient (\ref{e.rayquo})
\begin{align}
  \label{e.rayquomu}
  \mu(x):=\frac{(x,Bx)}{(x,Ax)}=\frac1{\rho(x)}.
\end{align}

\begin{lemma}
\label{l.basis}
Without loss of generality we can assume that $A=I$ and that $B=\diag(\mu_1,\ldots,\mu_n)$
with simple eigenvalues $\mu_1>\mu_2>\ldots>\mu_n>0$.
This transforms (\ref{e.psd}) (after multiplication with $\mu(x)=1/\rho(x)$
and by denoting the transformed preconditioner again by $T$)
in the form 
\begin{equation}
   \label{e.psdmu} 
      \mu(x)x' =\mu(x)x+\vartheta_\mathrm{opt} T(Bx-\mu(x)x)
\end{equation}
with the optimal step length
$$ \vartheta_\mathrm{opt}=\arg\max_{\vartheta\in\RR}\mu(\mu(x)x+\vartheta T(Bx-\mu(x)x) ).$$
The quality constraint (\ref{e.bass}) on the preconditioner $T\in\Rnn$
turns into a bound for the spectral norm $\| \cdot\|$ of the  
symmetric matrix $I-T$ which reads
\begin{align}
  \label{e.bass2}
  \|I-T\|\leq\gamma.
\end{align}
\end{lemma}
\begin{proof}
The generalized eigenvalue problem (\ref{e.mgevp}) is first transformed into a standard 
eigenvalue problem $C^{-1}BC^{-T}y=\mu y$ using the Cholesky factorization $A=CC^T$,
$y=C^Tx$ and $\mu=1/\lambda$. The symmetric matrix  $C^{-1}BC^{-T}$ can be
diagonalized by means of an orthogonal similarity transformation. Then all transformations are
applied to (\ref{e.psd}). For convenience we denote the transformed system matrix by $B$.
Further the transformed preconditioner is  denoted, once again, by $T$, since
(\ref{e.bass}) still holds with $A=I$. All this results in (\ref{e.psdmu}) and (\ref{e.bass2}).
 
To show that the proof of Thm.~\ref{t.psd} can be restricted
to the simple eigenvalue case we apply the same continuity argument which has been 
used in Theorem 2.1 in \cite{KNN2009}. 
The argument is based on a perturbation $B_\epsilon$
of $B$ having only simple eigenvalues. Then the perturbation $\epsilon$ is reduced 
to 0. The continuous dependence of $x'$ and $\mu(x')$ on the perturbation
completes the proof. This reasoning can be transferred to PSD since the Rayleigh-Ritz
procedure preserves the continuity of the eigenvalue approximations. 
\end{proof}

Next the reformulation of Thm.~\ref{t.psd} in terms of the $\mu$-notation is stated.
\begin{theorem}
  \label{t.psdmu}
  If $\mu_{i+1}<\mu(x)\leq\mu_i$ then $\mu(x')\geq\mu(x)$ and 
  either $\mu(x')\geq\mu_i$ or
  \begin{equation}
  \label{e.psdestmu}
  \begin{aligned}
    &\frac{\mu_i-\mu(x')}{\mu(x')-\mu_{i+1}}
    \leq\sigma^2
    \frac{\mu_i-\mu(x)}{\mu(x)-\mu_{i+1}}, \\[0.4em]
    &\text{with } \sigma=\frac{\kappa+\gamma(2-\kappa)}{(2-\kappa)+\gamma\kappa}
    \quad \text{ and } \quad
      \kappa=\frac{\mu_{i+1}-\mu_n}{\mu_i-\mu_n}.
  \end{aligned}
  \end{equation}
The estimate is sharp and can be attained for $\mu(x)\to\mu_i$ in the 3D invariant subspace 
associated with the eigenvalues $\mu_i$, $\mu_{i+1}$ and $\mu_n$, $i+1\neq n$.
\end{theorem}

\subsection{The cone of PSD iterates}

The starting point of the geometric description of PSD is the {\em non-scaled}
 preconditioned 
gradient iteration (\ref{e.preeig}) whose $\mu$-representation reads
\begin{equation}
  \label{e.pinvit}
  \begin{aligned}
     \mu(x) x'=\mu(x) x+T(Bx-\mu(x)x)
              =Bx-(I-T)(Bx-\mu(x)x).
  \end{aligned}
\end{equation} 
A central idea of its convergence analysis in
\cite{NEY2001a,NEY2001b,KNN2003} is to treat the preconditioners
{\em on the whole}. This means that all admissible preconditioners
satisfying the  spectral equivalence (\ref{e.bass2}) are inserted to
(\ref{e.pinvit}) with $x$ being fixed. This results in a 
set $\bgax$ of all possible iterates
\begin{align}
  \label{e.bgamma}
  \bgax:=\{Bx-(I-T)(Bx-\mu(x)x);\;  T \text{ s.p.d. with }
       \|I-T\|\leq\gamma\}. 
\end{align}
The set $\bgax$ is a full ball with the center $Bx$ and the radius 
$\gamma\|Bx-\mu(x)x\|$.
The subject of the convergence analysis of (\ref{e.pinvit})
in \cite{NEY2001a,NEY2001b} is to localize a vector 
of poorest convergence (i.e.~with the smallest Rayleigh quotient)
in $\mathcal{B}_\gamma(x)$ and to derive an estimate
for its Rayleigh quotient.

\begin{figure}[t]
\begin{minipage}[c]{0.45\textwidth}
  \begin{center}
   \includegraphics[height=4cm]{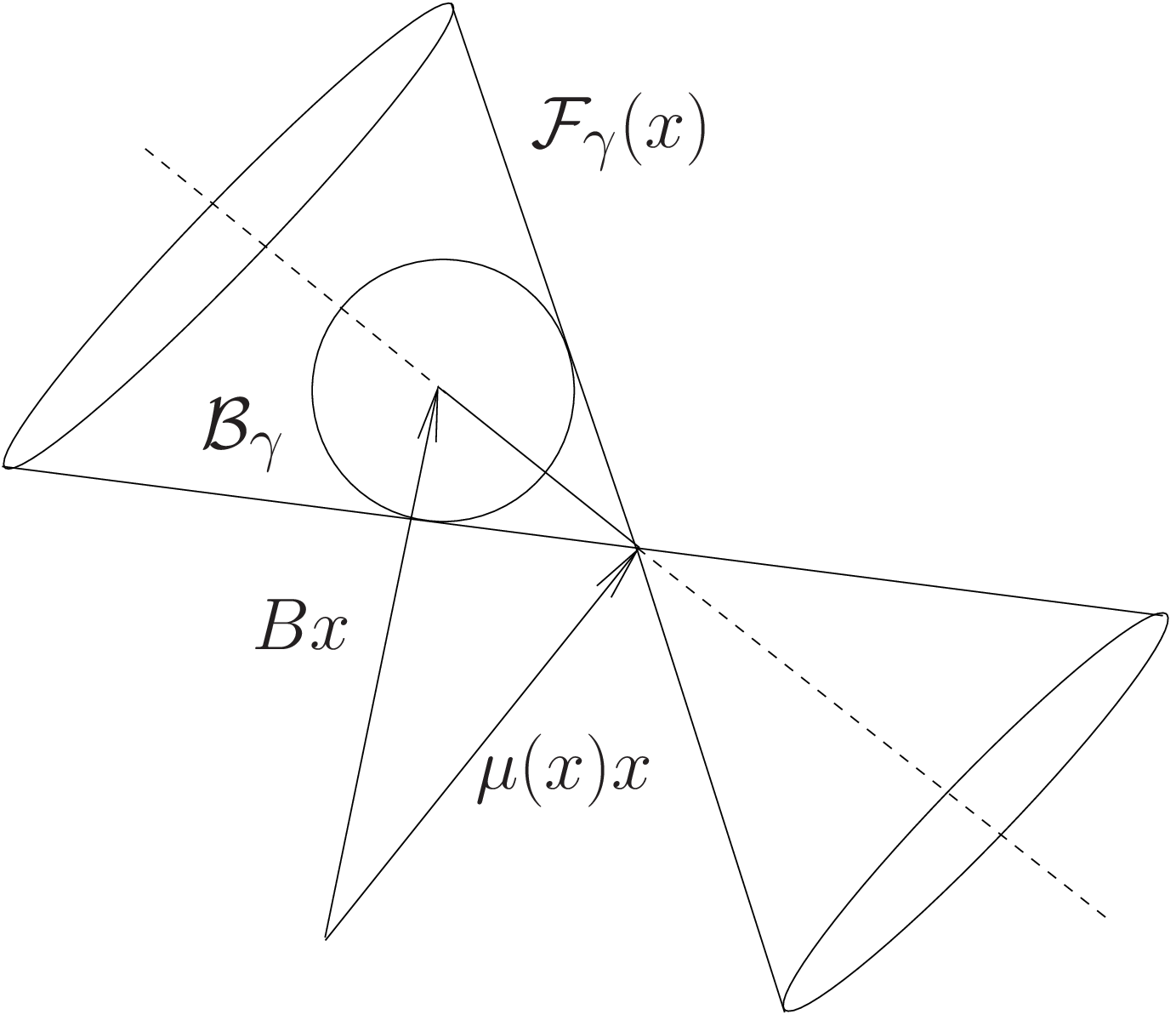}
  \end{center}
  \caption{The circular cone $\mathcal{F}_\gamma(x)$.}
  \label{f.kappa}
\end{minipage}\;
\begin{minipage}[c]{0.45\textwidth}
  \begin{center}
   \includegraphics[height=4cm]{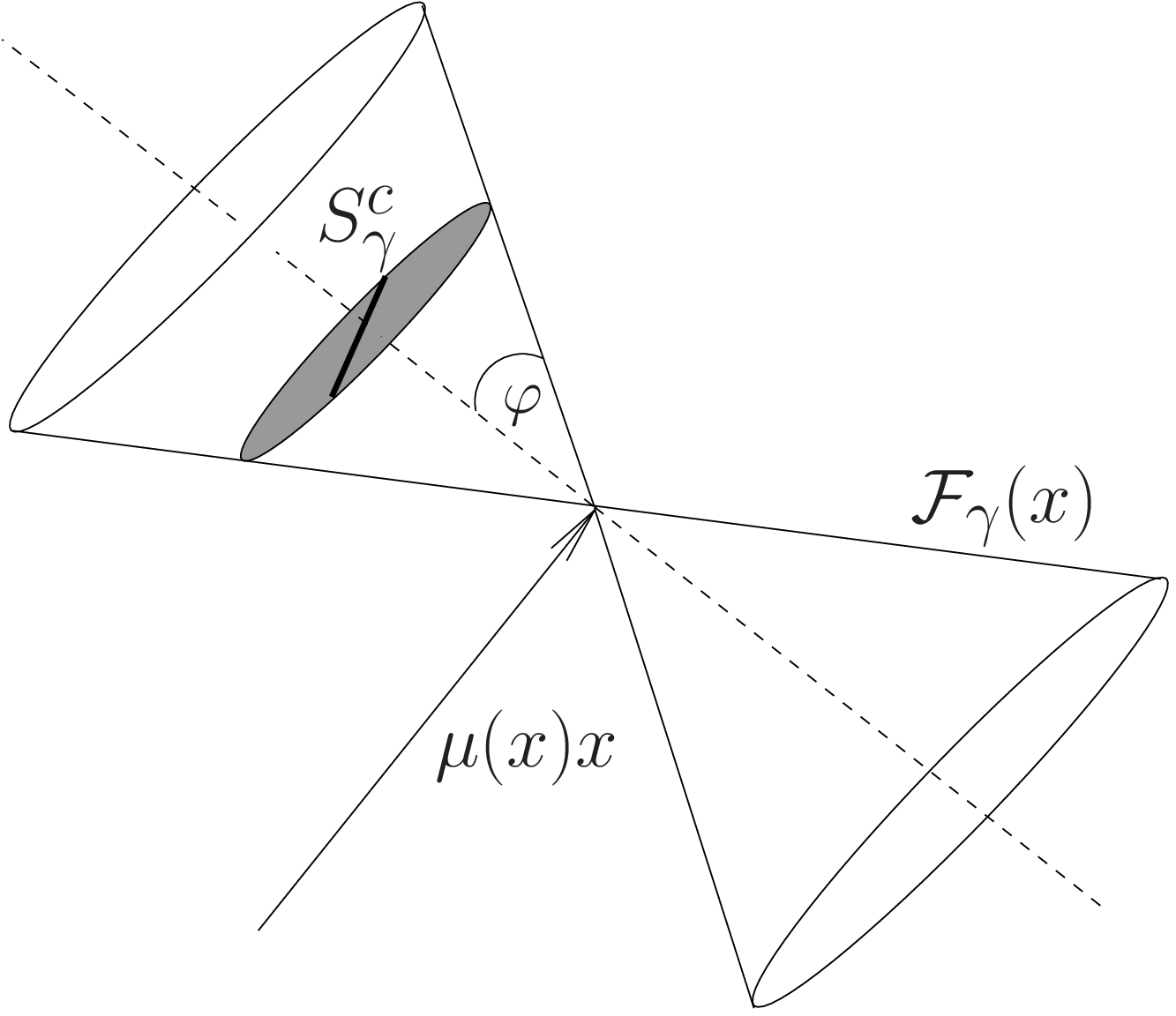}
  \end{center}
  \caption{The cross section $S_\gamma^c$ and the line segment
   $S_\gamma$ (bold line).}
  \label{f.2kappa}
\end{minipage}
\end{figure}

In contrast to (\ref{e.pinvit}) the PSD iteration (\ref{e.psdmu}) works with an optimal
step length parameter $\vartheta_\mathrm{opt}$ in order to  maximize the Rayleigh
quotient in the one-dimensional affine space 
\begin{align}
  \label{e.1Daff}
  \mu(x)x+\vartheta T(Bx-\mu(x)x),\qquad  \vartheta\in\RR.
\end{align}
The union of all these affine spaces for all the preconditioners 
satisfying (\ref{e.bass2}) is the smallest circular cone 
with its vertex in $\mu(x)x$ which encloses $\bgax$.
This cone is denoted by $\mathcal{F}_\gamma(x)$, see Fig.~\ref{f.kappa}, and it holds that
\begin{equation}
  \label{e.fgamma}
  \begin{aligned}
     \mathcal{F}_\gamma(x)&:=
        \{\mu(x)x + \vartheta(y-\mu(x)x);\; y\in\mathcal{B}_\gamma(x);\;
                \vartheta\in\RR\}\\
      &= \{ \mu(x)x+\vartheta d;\; \|Bx-(\mu(x)\,x+ d)\|\leq\gamma \|Bx-\mu(x)\, x\|; \; 
                        \vartheta\in\RR \}.          
  \end{aligned}
\end{equation}

\subsection{The geometric convergence analysis as a two-level optimization}
\label{ss.twolev}

The geometric convergence analysis of preconditioned steepest descent consists of
estimating the poorest convergence behavior.
Therefore a two-level optimization problem is to be solved. On the one hand one
has to determine this affine space (\ref{e.1Daff}) in the cone $\fgax$ in which
the maximum of the Rayleigh quotient (i.e.~the largest Ritz value in this space)
takes its smallest value;  this vector is associated
with the poorest convergence due to the choice of the preconditioner.
On the other hand the cone $\fgax$ depends on $x$; hence one can
analyze the dependence of this vector of poorest convergence 
on all vectors in the $\Rn$ having the same Rayleigh quotient as $x$. This amounts
to considering the level set of the Rayleigh quotient of
vectors having a fixed Rayleigh quotient $\mu_0$, i.e.
$$\lmun :=\{x\in\Rn;\; \mu(x)=\mu_0\}.$$
Let $x^*\in  \lmun$ be the minimizer representing the poorest convergence and let  
$d^*\in {\mathcal F}_\gamma(x)-\mu(x)x$  be the search direction of 
poorest convergence.
So the {\em two-level optimization} is
\begin{align*}
   \underline{\mu}:=\min_{x\in{\mathcal L}(\mu_0)} \;
             \min_{d\in{\mathcal F}_\gamma(x)-\mu_0 x} \mu(\mu_0 x+\vartheta_\text{opt}[x,d]d).
\end{align*}
Therein $\mu(x)x+\vartheta_\text{opt}[x,d]d$ is the Ritz vector which is associated with the 
larger Ritz value $\mu(x+\vartheta_{\text{opt}}[x,d]d)$ in $\spanl\{x,d\}$.
The factor $\vartheta_\text{opt}=\vartheta_\text{opt}[x,d]$ depends on $x$ and $d$.
The minimum $\underline{\mu}$ is now to be estimated from below.

\section{The level set optimization - a reduction to 3D}
\label{s.extconv}

The aim of this section is to show that the poorest convergence of PSD with
respect to the admissible preconditioners and with respect to all vectors $x\in\lmun$
is attained in a three-dimensional $B$-invariant subspace of the $\Rn$.

The representation (\ref{e.1Daff}) of the PSD iteration applies the 
line search to  $d\in\fgax-\mu(x)x$. This may result in an unbounded step length. 
To see this let $d=e_1=(1,0,\ldots,0)^T$ which is an eigenvector of $B$.
If $\gamma$ is close to 1, then $e_1\in\fgax-\mu(x)x$ can be attained since
$\lim_{\gamma\to 1}\fgax=\Rn$. The unboundedness is a consequence of
$\lim_{\vartheta\to\pm\infty}\mu(\mu(x)x+\vartheta e_1)=\mu_1$. 
The potential unboundedness of the step length has already been 
pointed out by Knyazev \cite{k87a}.

Next we want to avoid this singularity.  Therefore let $x'=\vartheta x+d$.
Due to $\mu(x')>\mu(x)$ (which is guaranteed 
by Thm.~\ref{t.pinres}) $\vartheta$ is bounded. So the minimization problem
is reformulated as
\begin{equation}
   \label{e.ritzvalmax}
   \underline{\mu}:=\min_{x\in{\mathcal L}(\mu_0)} 
             \min_{d\in{\mathcal F}_\gamma(x)-\mu_0 x} \mu(\vartheta_\text{opt}[x,d]x+d).
\end{equation}

In the next theorem a necessary condition characterizing this minimum is derived by means of the
Kuhn-Tucker conditions \cite{NOW2006}. 
The application of the Kuhn-Tucker conditions in the context of the convergence analysis
of the fixed-step size preconditioned gradient iteration has been suggested by R.~Argentati,
see \cite{AKNO2010}.

\begin{theorem}
  \label{t.3Djust}
   The minimum (\ref{e.ritzvalmax}) is attained in a three-dimensional 
   $B$-invari\-ant subspace of the $\Rn$. 
   
   If PSD does not terminate in an eigenvector, then the associated Ritz vector $w$ of
   poorest convergence is also contained in 
   the same three-dimensional  $B$-invari\-ant subspace of the $\Rn$, i.e.
   $$   (B+a)w=c(B+b)x$$
   with $a,b,c\in\RR$ and $B+a$ being a regular matrix.
\end{theorem}
\begin{proof}
The minimization problem (\ref{e.ritzvalmax}) reads as follows:
\begin{align*}
& \text{Minimize }\\
& \qquad \mu(\vartheta_\text{opt}x+d) \\
& \text{with respect to $x,d\in\Rn$ satisfying the two constraints:}\\
& \text{1. The cone inequality constraint $d\in{\mathcal F}_\gamma(x)-\mu_0 x$}\\
& \qquad g(x,d)=\|Bx-(\mu_0 x+d)\|^2-\gamma^2\|Bx-\mu_0 x\|^2\\
&\phantom{g(x,d)\qquad} =(1-\gamma^2)\|Bx-\mu_0 x\|^2-2(Bx-\mu_0 x,d)+\|d\|^2\leq 0.\\
& \text{2. The level set constraint $x\in{\mathcal L}(\mu_0)$}\\
& \qquad h(x,d)= (x,Bx)-\mu_0 (x,x)=0.
\end{align*}
Therein $\vartheta_\text{opt}=\vartheta_\text{opt}[x,d]\in\RR$ is a functional depending on 
$x$ and $d$ which maximizes the Rayleigh quotient in the two-dimensional subspace
$\spanl\{x,d\}$. Equivalently $w:=\vartheta_\text{opt}x+d$ is a Ritz vector corresponding
to the larger Ritz value in just this two-dimensional subspace. The first constraint 
guarantees that $d$ is an admissible search direction, i.e.~the distance of $\mu_0x+d$ to the
center $Bx$ of the ball $\bgax$ is bounded by its radius $\gamma\|Bx-\mu_0 x\|$.

The Karush-Kuhn-Tucker stationarity condition for a local minimizer
 $(x^*,d^*)$ reads
$$  \nabla_{(x,d)}\mu(\vartheta_\text{opt}x^*+d^*)+\alpha\nabla_{(x,d)}g(x^*,d^*)
                +\beta\nabla_{(x,d)}h(x^*,d^*)=0$$
with the multipliers $\alpha$ and $\beta$.		
In order to simplify the notation, the asterisks are omitted from now on.

Next we derive the gradients of these functions $\mu$, $g$ and $h$ with respect to
$x$ and $d$. The chain rule gives (for column vectors)
$$  \nabla_x \Big(\mu(\vartheta_\text{opt}x+d)\Big)= \left(D_x(\vartheta_\text{opt}x+d)\right)^T
   (\nabla\mu)(\vartheta_\text{opt}x+d).$$
It holds that
$$   \left( D_x(\vartheta_\text{opt}x+d) \right)_{ij}  = (x(\nabla_x \vartheta_\text{opt})^T
   +\vartheta_\text{opt} I)_{ij}.$$
With $w:=\vartheta_\text{opt}x+d$ we get
\begin{align*}
    \nabla_x\Big(\mu(\vartheta_\text{opt}x+d)\Big)&= \vartheta_\text{opt} (\nabla \mu)(w)+
     (\nabla_x \vartheta_\text{opt}) \;   (x, (\nabla \mu)(w))\\
      &=\vartheta_\text{opt} (\nabla \mu)(w)
      =\vartheta_\text{opt}\frac2{(w,w)}(Bw-\mu(w)w).
\end{align*}
Therein, $(x, (\nabla \mu)(w))=0$ has been used which holds since
$(\nabla \mu)(w)$ is collinear to the residual of the Ritz vector and further, 
by definition of a Ritz vector, its residual is orthogonal 
to the approximating subspace $\spanl\{x,d\}$.   For the $d$-gradient it holds that
$$  \nabla_d\Big(\mu(\vartheta_\text{opt}x+d)\Big)=(\nabla \mu)(w)
        =\frac2{(w,w)}(Bw-\mu(w)w).$$
The gradients of the constraining functions $g$ and $h$ with $r=Bx-\mu_0 x$
are 
\begin{alignat*}{2}
    \nabla_x g(x,d)& = (1-\gamma^2)2(B-\mu_0)r-2(B-\mu_0)d,\qquad& \nabla_x h(x,d)&= 2r,\\
    \nabla_d g(x,d)&=-2(B-\mu_0) x+2d =2(d-r),&   \nabla_d h(x,d)&= 0.
\end{alignat*}

Hence the $x$-components of the Karush-Kuhn-Tucker stationarity condition are
\begin{align}
   \label{e.xgradkkt}
   \frac{\vartheta_\text{opt}}{(w,w)}(B-\mu(w))w+\alpha \Big\{ (1-\gamma^2)(B-\mu_0)^2x-
  (B-\mu_0)(w-\vartheta_\text{opt} x)\Big\} + \beta r=0
\end{align}
and the $d$-components read $(Bw-\mu(w)w)+\alpha (w,w)(d-r)=0.$
The equation for the $d$-components can be reformulated as
\begin{align}
  \label{e.bawx}
  (B+a)w=\alpha (w,w)(B+b)x
\end{align}
with $a=\alpha (w,w)-\mu(w)$ and $b=\vartheta_\text{opt}-\mu_0$. 
Multiplication of (\ref{e.xgradkkt}) with $B+a$ and insertion of (\ref{e.bawx})
results in
\begin{align*}
  &\alpha \Big\{(1-\gamma^2)(B-\mu_0)^2(B+a)x-(B-\mu_0)\left[\alpha (w,w)(B+b)x
    -\vartheta_\text{opt}(B+a)x\right]\Big\}\\
   & +\alpha\vartheta_\text{opt}(B-\mu(w))(B+b)x + \beta (B+a)(B-\mu_0)x=0.
\end{align*}
This can be expressed as
\begin{equation}
   \label{e.p3b}
    p_3(B)x=0
\end{equation}
with a third order polynomial $p_3$. Due to the basis assumptions $B$ is a diagonal matrix and
so $p_3(B)$ is diagonal. As $p_3$ has at most three different zeros, (\ref{e.p3b}) can 
only hold if $x$ has at most three non-zero components, which proves the first assertion.

Hence $x\in\spanl\{e_j,e_k,e_l\}$ for proper indexes $j$, $k$ and $l$.
For that $x$ Eq.~(\ref{e.bawx}) shows that $w$ has not more than four non-zero components;
four non-zero components are only possible if $a=-\mu_s$ for $s\neq j,k,l$. Then 
(\ref{e.xgradkkt}) can be written as $p_1(B)w=p_2(B)x\in\spanl\{e_j,e_k,e_l\}$ with
a first order polynomial $p_1$ and a second order polynomial $p_2$. The latter equation
implies that $p_1(\mu_s)=p_1(-a)=0$. The $s$-th component of the polynomial identity
results in $a=(\alpha\mu_0(w,w)-\mu(w)\vartheta_\text{opt})/(\vartheta_\text{opt}-\alpha(w,w))$.
Together with the known form $a=\alpha (w,w)-\mu(w)$ we get by direct computation that $a=b$.
Insertion of this result to (\ref{e.bawx}) shows that 
$w=\alpha(w,w)x+Ce_s$ for a real constant $C$. Then $x\perp e_s$ and $x$ and $e_s$ are
the Ritz vectors. PSD terminates in 
$e_s$ and $w$ with not more than three non-zero components is the normal case.
\end{proof}

\section{The cone optimization - a mini-dimensional geometric analysis}
\label{s.mini}

Next the convergence behavior with respect to the cone $\fgax$ is analyzed. 
Some of the following arguments are valid in the $\Rn$; however we need 
these properties only for $n=3$.

The (half) opening angle $\varphi$ of the cone $\mathcal{F}_\gamma(x)$ is 
given by $\sin\varphi=\gamma$, since $\gamma$ is the ratio of the radius 
$\gamma \|Bx-\mu(x)x \|$ of
the ball $\bgax$, see (\ref{e.bgamma}), and its (maximal) radius 
$\|Bx-\mu(x)x \|$ for $\gamma\to 1$.
With $\cos\varphi=\sqrt{1-\gamma^2}$ the cone $\mathcal{F}_\gamma(x)$
can be written as 
$$  \mathcal{F}_\gamma(x):=\mu(x)x+\{ z\in\RR^n;\; 
    \Big| (\frac{z}{\|z\|},\frac{Bx-\mu(x)x }{\|Bx-\mu(x)x \|})\Big|\geq
       \sqrt{1-\gamma^2} \}.$$

\subsection{Restriction to non-negative vectors}
\label{ss.restrict}

The analysis of PSD can be restricted to
component-wise non-negative vectors $x\in\Rn$. The justification is as follows. 
Consider the
Householder reflections $H_i=I-2e_ie_i^T$ for which
$x\mapsto H_i x$ changes the sign of the $i$th component of $x$. 
The Rayleigh quotient is invariant under $H_i$, i.e.~$\mu(x)=\mu(H_i x)$. 
If $v$ is an admissible search direction, i.e.~$v\in\fgax-\mu(x)x$, then 
\begin{align*}
   \cos\measuredangle(v,Bx-\mu(x)x)&=(\frac{v}{\|v||},\frac{Bx-\mu(x)x}{\|Bx-\mu(x)x\|})
    = (\frac{H_iv}{\|H_iv||},\frac{BH_ix-\mu(H_ix)H_ix}{\|BH_ix-\mu(H_ix)H_ix\|})\\
    &= \cos\measuredangle(H_iv,BH_i x-\mu(H_ix)H_i x),
\end{align*}
which means that $H_iv$ encloses the same angle with the residual vector
associated with $H_ix$. As for all $\alpha\in\RR$ 
$$ \mu(\mu(H_ix)H_ix+\alpha H_iv)=\mu(H_i(\mu(x)x+\alpha v))=\mu(\mu(x)x+\alpha v)$$
any Rayleigh quotient in the cone $\mathcal{F}_\gamma(x)$ can be
reproduced in the cone $\mathcal{F}_\gamma(H_i x)$ and vice versa.
Thus the analysis can be restricted to $x\geq 0$.

\begin{figure}[t]
\begin{minipage}[c]{0.47\textwidth}
  \begin{center}
   \includegraphics[width=0.95\textwidth]{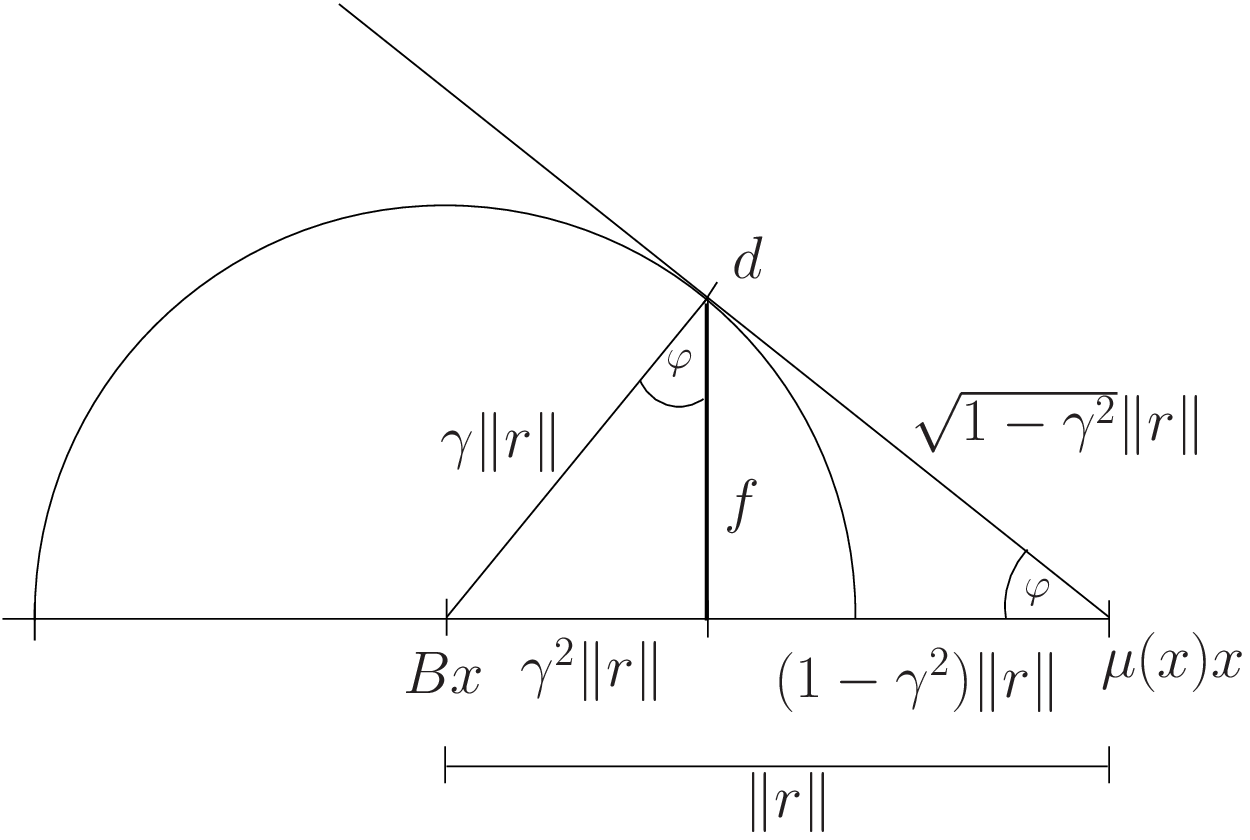}
  \end{center}
  \caption{3D-geometry with $r=Bx-\mu(x)x$, $f=\gamma \sqrt{1-\gamma^2}\,\|r\|$.}
  \label{f.3lg}
\end{minipage}\;
\begin{minipage}[c]{0.47\textwidth}
  \begin{center}
   \includegraphics[width=0.55\textwidth]{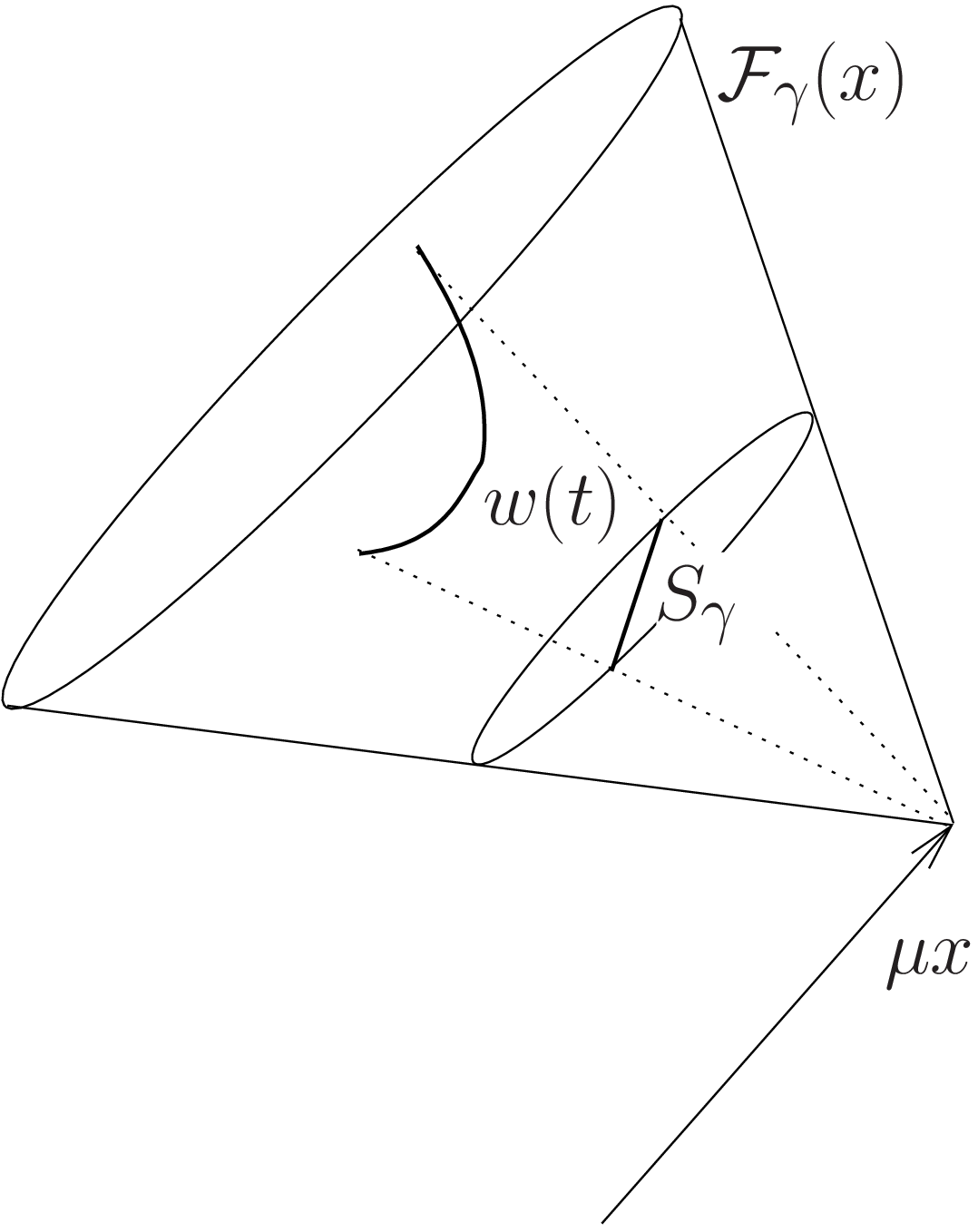}
  \end{center}
  \caption{The line segment $S_\gamma(x)$ and the curve $w(t)$.}
  \label{f.spperp}
\end{minipage}
\end{figure}

\subsection{The poorest convergence in the three-dimensional cone $\fgax$}
       
Any circular cross section $S_\gamma^c$ (with non-zero radius) of 
$\fgax$ can serve to represent the admissible search directions, 
see Fig.~\ref{f.2kappa}. Next we work with the disc
\begin{equation}
  \label{e.fgamma0}
    S_\gamma^c(x):= \mu(x)x+
      (1-\gamma^2)r+\{fy;\; y\in\RR^3, \;
       \|y\|\leq 1,\; y\perp r\}
\end{equation}
with $r:=Bx-\mu(x)x$. Its radius $f$, see Fig.~\ref{f.3lg}, is given by
\begin{align}
  \label{e.f}
  f=\gamma\sqrt{1-\gamma^2} \|r\|.
\end{align}
Further we use only search directions $d\in S_\gamma^c(x)-\mu(x)x$ which 
are orthogonalized against $x$; this is justified since 
the Rayleigh-Ritz approximations (and so the PSD iterate $x'$) only 
depend on the subspace. 
So the set of relevant search directions forms a line segment.
By using the vector $v=x\times r/\| x\times r\|=x\times r/(\|x\|\,\| r\|)$
one can construct the intersection of this line segment 
with the surface of the cone. The points of intersection are
$d_{1/2}$ with
\begin{align}
  \label{e.d1}
   d_1&=\mu(x)x+(1-\gamma^2)r + \gamma \sqrt{1-\gamma^2}\,\|r\|v,\\
  \label{e.d2}
   d_2&=\mu(x)x+(1-\gamma^2)r - \gamma \sqrt{1-\gamma^2}\,\|r\|v, \\ 
   v &=\frac{x\times r}{\|x\|\,\|r\|} . \notag
\end{align}
Therefore the line segment has the form  (see Fig.~\ref{f.spperp})
\begin{align}
  \label{e.lseg}
  S_\gamma(x):=\{d(t):=t d_1+(1-t) d_2; \; t\in[0,1]\}.
\end{align}

\begin{lemma}
The poorest convergence of PSD in 3D (aside from the singular cases 
that PSD terminates in an eigenvector) is attained in $d_1$ or $d_2$
as given by (\ref{e.d1}) and (\ref{e.d2}).

\end{lemma}
\begin{proof}
The line segment $S_\gamma$ has the form $d(t)$ with $t\in[0,1]$
by (\ref{e.lseg}). The PSD iteration
maps $S_\gamma$ into a curve $w(t)$, $t\in[0,1]$, where $w(t)$ is the 
Ritz vector $w(t)=\mu(x)x+\vartheta_\text{opt}(t)d(t)$ corresponding to 
the larger Ritz value in $\spanl\{x,d(t)\}$. 
(A singularity like that mentioned at the beginning of Sec.~\ref{s.extconv} 
has not to be considered since otherwise the first alternative $\mu(x')\geq\mu_i$
in Thm.~\ref{t.psdmu} applies and nothing is to be proved.)
Along $w(t)$ we are looking for a vector $w^*=w(t^*)$ so that 
\begin{align*}
   \mu(w(t^*))\leq \mu(w(t)) \quad\forall t\in [0,1].
\end{align*}
Since $w(t)$ is a Ritz vector its residual $Bw(t)-\mu(w(t))w(t)$ is orthogonal 
to the subspace spanned by $x$ and $d(t)$. As the residual is collinear
to the gradient vector $\nabla\mu(w(t))$ we get
\begin{align}
  \label{e.pccond1}
  (\nabla\mu(w(t)),\spanl\{x,d(t)\})=0.
\end{align}
A stationary point of the Rayleigh quotient in a $t\in(0,1)$ is attained if
\begin{align*}
   0&=\frac{d}{dt}\mu(w(t))=(\nabla\mu(w(t)),w'(t))\\
    &=( \nabla\mu(w(t)),\vartheta_\text{opt}'(t)d(t)
          + \vartheta_\text{opt}(t)d'(t))\\
    &= ( \nabla\mu(w(t)),\vartheta_\text{opt}(t)d'(t))
\end{align*}
where (\ref{e.pccond1}) has been used for the last identity. As 
$d'(t)$ is collinear to $x\times r$ we get from 
$(\nabla\mu(w(t)),d'(t))=0$ together with (\ref{e.pccond1}) that $\nabla\mu(w)=0$
(since $x$, $d$ and $d'$ span the $\RR^3$).
So any interior stationary point must be an eigenvector
and hence $\mu(w(t))$ take  the other extrema on the surface for $t=0$ or $t=1$
in $d_1$ or $d_2$. 
\end{proof}
\smallskip

Next we apply the Rayleigh-Ritz procedure to the two-dimensional subspaces
$[x,d_i-\mu(x)x]$, $i=1,2$, in order to determine whether  
the poorest convergence is attained in $d_1$ or $d_2$. 
First the Euclidean norm of $d_i-\mu(x)x$ is determined
\begin{align*}
  \|d_i-\mu(x)x\|^2 =& (1-\gamma^2)^2(r,r)\pm(1-\gamma^2)\gamma \sqrt{1-\gamma^2}
    (r,x\times r)/\|x\|\\
     &+\gamma^2(1-\gamma^2)\|x\times r\|^2/\|x\|^2\\
     =& (1-\gamma^2)^2\|r\|^2+\gamma^2(1-\gamma^2)\|r\|^2
     =(1-\gamma^2)\|r\|^2.
\end{align*}
Hence the normalized search directions $(d_i-\mu(x)x)/\| d_i-\mu(x)x \|$ are
\begin{align*}
  \bar{d}_{1/2}:=\frac{d_{1/2}-\mu(x)x}{\sqrt{1-\gamma^2}\|r\|}
  =\sqrt{1-\gamma^2}\,\frac{r}{\|r\|} \pm\gamma\, \frac{x\times r}{\|x\|\,\|r\|}
\end{align*}
and therefore $V_1=[x,\bar{d}_1]$ and $V_2=[x,\bar{d}_2]\in\RR^{3\times 2}$  
are orthonormal matrices.
The Ritz values of $B$ in the column space of $V_i$ are the eigenvalues of the
projection
\begin{align*}
  B_i:=V_i^TBV_i=\left(\begin{array}{cc}\mu(x) & (\bar{d}_i,Bx)\\ 
   (\bar{d}_i,Bx) & \mu(\bar{d}_i)\end{array}
          \right).
\end{align*}
The larger Ritz value (that is the larger eigenvalue of $B_i$) reads 
\begin{align}
  \label{e.th2i}
  \theta_{2,i}=
    \frac{\mu(x)+\mu(\bar{d}_i)}{2}+\sqrt{ \frac{(\mu(x)-\mu(\bar{d}_i))^2}{4}+
   (\bar{d}_i,Bx)^2}.
\end{align}   
In order to decide whether in $d_1$ or in $d_2$ poorest convergence is taken, we
show that the non-diagonal elements of $B_i$ do not depend on $i$ since
\begin{align}
   \label{e.diBx}
   (\bar{d}_i,Bx)=(\bar{d}_i,Bx-\mu(x)x)=\|r\|(\bar{d}_i,\frac{r}{\|r\|})
        =\|r\|\cos\measuredangle(\bar{d}_i,r)=\sqrt{1-\gamma^2}\|r\|.
\end{align}
Hence only the (2,2) element of $B_i$ depends on $i$.  As further
\begin{align*}
   \frac{d \theta_{2,i}}{d\mu(\bar{d}_i)}=\frac12\Bigg(
   1-\frac1{1+\left(\frac{2 (\bar{d}_i,Bx)}{\mu(x)-\mu(\bar{d}_i)  }\right)^{1/2}}\Bigg)>0
\end{align*}
shows that $\theta_{2,i}$ is a monotone increasing function of $\mu(\bar{d}_i)$ we
still have to find the $d_i$ with the smaller Rayleigh quotient in order to find
the search direction which is associated with the poorer PSD convergence.

\begin{lemma}
\label{l.pcot}
PSD in 3D takes its poorest convergence, i.e.~the smallest value of $\theta_2$,
in
\begin{align}
   \label{e.swd}
   d=\mu(x)x+(1-\gamma^2)r+\gamma\,\sqrt{1-\gamma^2} \frac{x\times r}{\|x\|},
\end{align}
if $x\in\Rn$ is a component-wise non-negative vector (cf.~Sec.~\ref{ss.restrict}).
The associated Ritz value is accessible from (\ref{e.th2i}).
\end{lemma}

\begin{proof}
We show that $\theta_{2,1}$ is the smaller Ritz value by showing (we use the
monotonicity of $ \theta_{2,i}[\mu(\bar{d}_i)]$) that
$ \mu(\bar{d}_1)\leq \mu(\bar{d}_2)$. This inequality is true if
$(r,B(x\times r))\leq 0$.
By using $\mathrm{span}\{x,r\}\perp x\times r$ and $r\perp x$ direct computation results in
\begin{align*}
  (r,B(x\times r))&= (B(Bx-\mu(x)x),x\times r)=(B^2x,x\times r)-\mu(x)(Bx,x\times r)\\
  &= (B^2x,x\times r)-\mu(x)(r+\mu(x)x,x\times r)\\
  &=(B^2x,x\times r)=(r,B^2x\times x)=(Bx,B^2x\times x)\\
  &=-x_1x_2x_3(\mu_1-\mu_2)(\mu_1-\mu_3)(\mu_2-\mu_3)\leq 0.
\end{align*}
The last inequality holds since $x\geq 0$ and $\mu_1>\mu_2>\mu_3$. 
\end{proof}

\subsection{A mini-dimensional convergence analysis of PSD}
\label{ss.levset}

Due to Thm.~\ref{t.3Djust} the ``mini-dimensional'' convergence analysis 
can be restricted to three-dimensional $B$-in\-vari\-ant subspaces of the $\Rn$. 
With respect to the basis of eigenvectors these subspaces have the form 
$\spanl\{e_j,e_k,e_l\}$ where $e_*$ is the $*$-th unit vector. The associated
eigenvalues are indexed so that $\mu_j>\mu_k>\mu_l$.

Lemma \ref{l.pcot} delivers for any $x\in\lmu$ in 3D the vector of $\bgax$-poorest 
PSD convergence. Next we have to analyze the $\lmu$-dependence of the poorest
convergence case.

\begin{figure}[t]
\begin{minipage}[c]{0.44\textwidth}
  \begin{center}
   \includegraphics[width=0.9\textwidth]{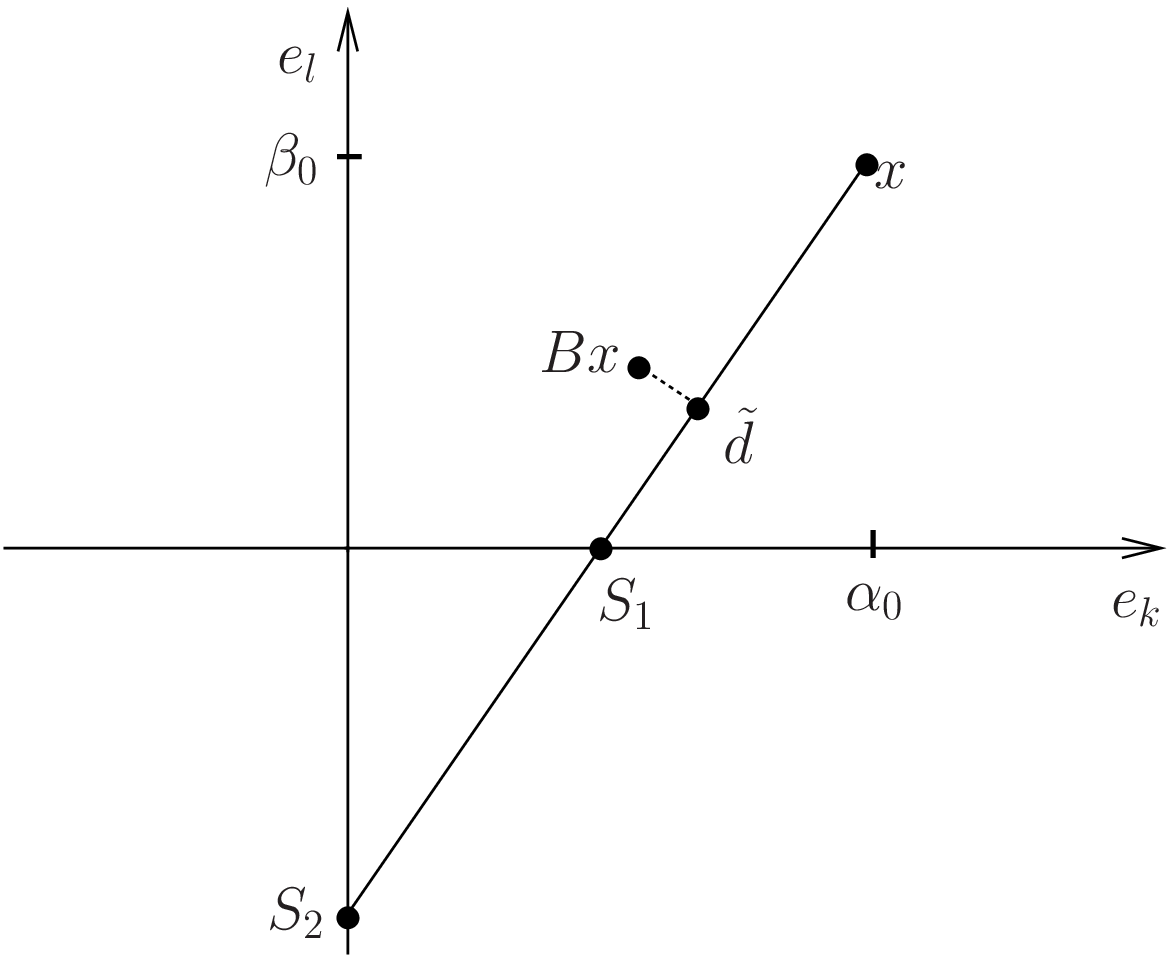}
  \end{center}
  \caption{Geometry in the plane $\mathcal{E}_j$ and $Bx\notin\mathcal{E}_j$.}
  \label{f.planeE}
\end{minipage}\hspace{0.5cm}
\begin{minipage}[c]{0.44\textwidth}
  \begin{center}
   \includegraphics[width=0.9\textwidth]{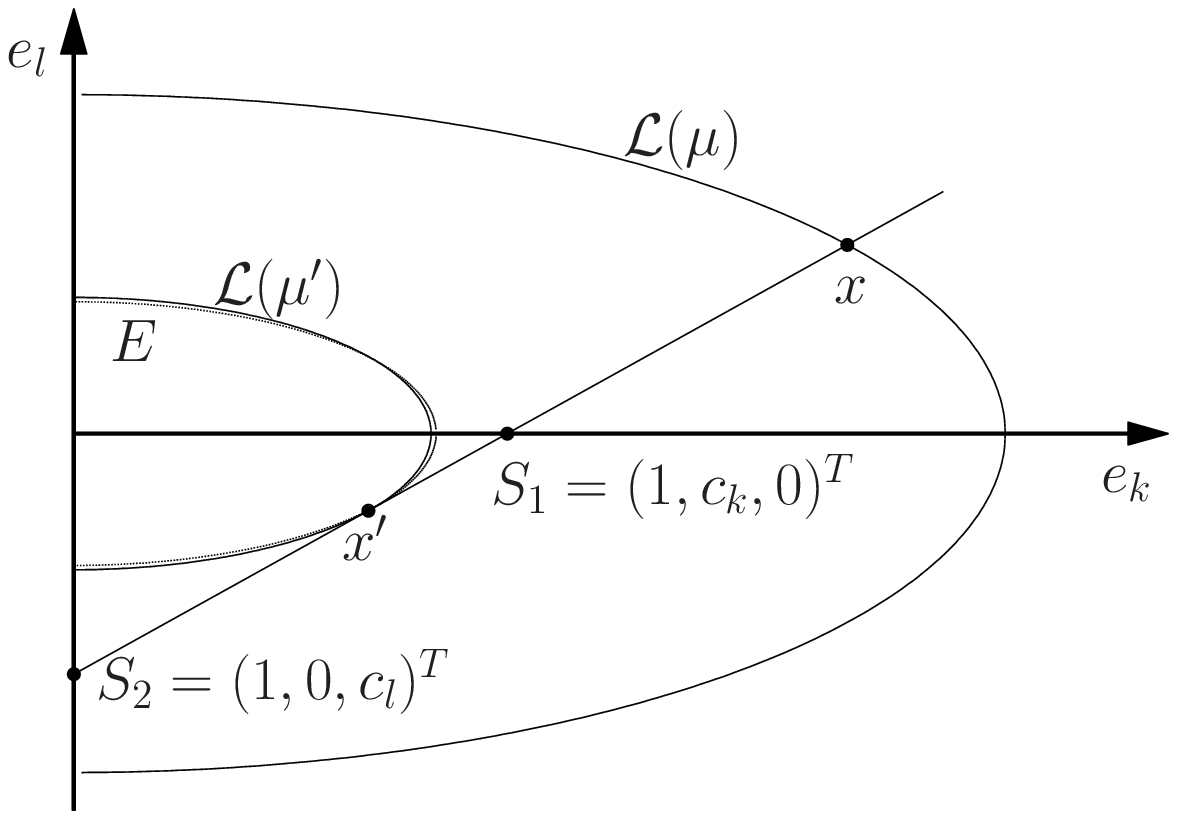}
  \end{center}
  \caption{Ellipses in $\mathcal{E}_j$; $E$ and 
       ${\mathcal L}(\mu')\cap\mathcal{E}_j$ are almost identical}
  \label{f.ellip}
\end{minipage}
\end{figure}

\begin{theorem}
  \label{t.3Dest}
In the three-dimensional space $\spanl\{e_j,e_k,e_l\}$ the following sharp estimate for
PSD holds
$$    \frac{\Delta_{j,k}(\mu')}{\Delta_{j,k}(\mu)}\leq \left(
      \frac {\kappa+\gamma(2-\kappa)}{(2-\kappa)+\gamma\kappa}\right)^2$$
with 
$$ \Delta_{j,k}(\xi)=\frac{\mu_j-\xi}{\xi-\mu_k}
  \quad \text{and} \quad \kappa=\frac{\mu_k-\mu_l}{\mu_j-\mu_l}.$$

\end{theorem}
\begin{proof}
The starting point of the following analysis are the vectors $x$ and
$$   d=\mu(x)x+(1-\gamma^2)r+\gamma\sqrt{1-\gamma^2} \frac{x\times r}{\|x\|}.$$ 
Without loss of generality $x$ can be normalized in a way that
$$ x=e_j+\alpha_0e_k+\beta_0e_l;$$
hence $x$ is an element of the affine space $\mathcal{E}_j:=e_j+\spanl\{e_k,e_l\}$.
The coordinate  form of $x$ in 3D then is $x=(1,\alpha_0,\beta_0)^T$. 
Further let $\tilde{d}=(1,\tilde\alpha,\tilde\beta)^T\in\mathcal{E}_j$ 
the corresponding multiple of $d$. Since $\spanl\{x,d\}$ is a tangential plane
of the ball $\bgax$ in $d$ and $Bx-d$ is a radius vector of the ball it holds that
\begin{align}
  \label{e.bxperp}
    Bx-d\perp \spanl\{x,d\}=\spanl\{x,\tilde{d}\}.
\end{align}
Hence $Bx-d$ is collinear to
$$    x\times \tilde{d}=(\alpha_0\tilde\beta-\tilde{\alpha}\beta_0,
                          \beta_0-\tilde{\beta},\tilde\alpha-\alpha_0)^T.$$
By $S_1=(1,c_k,0)^T$ and $S_2=(1,0,c_l)^T$ with $S_1,S_2\in \mathcal{E}_j$ 
we denote the points of intersection of $\spanl\{x,\tilde{d}\}$
with $e_j+\spanl\{e_k\}$ and $e_j+\spanl\{e_l\}$, see Fig.~\ref{f.planeE}.
Due to (\ref{e.bxperp}) it holds that $(Bx-d, S_i)=0$, $i=1,2$. Since 
$$   Bx-d=\gamma^2r-\gamma\sqrt{1-\gamma^2}\, \frac{x\times r}{\|x\|}$$
we get with 
\begin{align*}
   r=\left( \begin{array}{c} \mu_j-\mu \\ (\mu_k-\mu)\alpha_0 \\ 
                 (\mu_l-\mu)\beta_0 \end{array} \right),\qquad
   x\times r = \left( \begin{array}{c} 
       \alpha_0\beta_0(\mu_l-\mu_k)\\ \beta_0(\mu_j-\mu_l) \\
       \alpha_0(\mu_k-\mu_j)
    \end{array} \right)
\end{align*}
from $(Bx-d,S_1)=0$ that
\begin{align}
  \label{e.ck2}
   c_k=  -\frac{(Bx-d)\!\mid_1}{(Bx-d)\!\mid_2}=
        \frac{\|x\|(\mu_j-\mu)+\Gamma\alpha_0\beta_0(\mu_k-\mu_l)}
             {\|x\|\alpha_0(\mu-\mu_k)+\Gamma\beta_0(\mu_j-\mu_l)}.
\end{align}
Analogously $(Bx-d,S_2)=0$ results in
\begin{align}
  \label{e.cl2}
    c_l=-\frac{(Bx-d)\!\mid_1}{(Bx-d)\!\mid_3}=
        \frac{\|x\|(\mu_j-\mu)+\Gamma\alpha_0\beta_0(\mu_k-\mu_l)}
             {\|x\|\beta_0(\mu-\mu_l)+\Gamma\alpha_0(\mu_k-\mu_j)}
\end{align}
with $\Gamma=\sqrt{1-\gamma^2}/\gamma$.

Any $x\in\mathcal{E}_j\cap\lmu$ is an element of the ellipse $(x_k/a)^2+(x_l/b)^2=1$
with
$$  a=\sqrt{\frac{\mu_j-\mu}{\mu-\mu_k}}, \qquad b=\sqrt{\frac{\mu_j-\mu}{\mu-\mu_l}}.$$
As justified in Sec.~\ref{ss.restrict} the analysis can be restricted to componentwise non-negative
$x=(1,\alpha_0,\beta_0)^T$ so that its components 
$\alpha_0$ and $\beta_0$ can be represented in terms of
$\psi\in(0,\pi/2)$ and $t=\tan\psi$
\begin{align}
  \label{e.al0be0}
    \alpha_0=a\cos(\psi)=a\sqrt{\frac{1}{1+t^2}}, \qquad
    \beta_0=b\sin(\psi)=b\sqrt{\frac{t^2}{1+t^2}}.
\end{align}
Two further ellipses in $\mathcal{E}_j$ are relevant for the subsequent analysis.
These ellipses are very similar, each centered in $e_j$ 
(the origin of $\mathcal{E}_j$) and each tangential
to the line through $S_1$ and $S_2$. The first ellipse is
$\mathcal{E}_j\cap\mathcal{L}(\mu')$ with $\mu'=\mu(x')$ and has the semi-axes
$$  a'=\sqrt{\frac{\mu_j-\mu'}{\mu'-\mu_k}}, \qquad 
    b'=\sqrt{\frac{\mu_j-\mu'}{\mu'-\mu_l}}.  $$
This ellipse is tangential to the line through $S_1$ and $S_2$ since $\mu(x')$ 
is associated with the poorest convergence on the cone $\fgax$ projected to
$\mathcal{E}_j$. Direct computation shows that
$a'/b'<a/b$. 

The second ellipse $E$, see Fig.~\ref{f.ellip}, has the semi-axes 
$\tilde{a}$ and $\tilde{b}$ so that the ratio of its semi-axes equals that
of  $\mathcal{E}_j\cap\lmu$. This means that $\tilde{a}/\tilde{b}=a/b$. It holds that
$\tilde{a}\geq a'$, since otherwise a contradiction can be derived. Assuming $\tilde{a}< a'$
for any point $(\alpha,\beta)$ on the ellipse $E$ it holds that (by using $a'/b'<a/b$)
$$  \alpha^2+\frac{a'^2}{b'^2}\beta^2<\alpha^2+\frac{a^2}{b^2}\beta^2=
    \alpha^2+\frac{\tilde{a}^2}{\tilde{b}^2}\beta^2=\tilde{a}^2<a'^2$$
so that $\alpha^2/a'^2+\beta^2/b'^2<1$. The latter inequality means that the ellipse $E$
is completely surrounded by the ellipse ${\mathcal L}(\mu')\cap {\mathcal E}_j$, which contradicts
its tangentiality to the line through $S_1$ and $S_2$. Hence 
$$\Delta(\mu')=\frac{\mu_j-\mu'}{\mu'-\mu_k}=a'^2\leq \tilde{a}^2$$
and an upper limit for $\tilde{a}^2/\Delta(\mu)=\tilde{a}^2/a^2$ remains to be determined.
Next we show that (the case $c_l\to\infty$ is to be treated separately by analyzing
the limits of $c_k$ and $c_l$)
\begin{align}
  \label{e.axeq}
  \frac{\tilde{a}^2}{a^2}=\frac{c_k^2c_l^2}{b^2c_k^2+a^2c_l^2}.
\end{align}
To prove this we determine the point of contact of the line through 
$S_1$ and $S_2$ and the ellipse $E$. The semi-axes of $E$ are $\tilde{a}$
and $\tilde{b}=b\tilde{a}/a$. By a rescaling of the second semi-axis with 
the factor $a/b$ the ellipse becomes a circle with the radius $\tilde{a}$ and
the point of contact does not change. Further the line segment 
connecting $S_1$ and $S_2$ is transformed 
$$  s(\sigma)=\left(\begin{array}{c} 0 \\ \frac{a}{b}c_l \end{array}\right)
   + \sigma \left(\begin{array}{c} c_k \\ -\frac{a}{b}c_l \end{array}\right),
   \qquad \sigma\in[0,1].$$
The point of contact is that point on $s(\sigma)$ with the smallest
Euclidean norm. From 
$$   \|s(\sigma)\|^2=\sigma^2 c_k^2+(\frac{a}{b}c_l)^2(\sigma-1)^2$$
direct computation shows that the minimum is attained in
$\sigma^*=a^2c_l^2/(b^2c_k^2+a^2c_l^2).$ The resulting identity
$\tilde{a}^2=\|s(\sigma^*)\|^2$ yields (\ref{e.axeq}).

Insertion of (\ref{e.ck2}), (\ref{e.cl2}) and (\ref{e.al0be0}) in (\ref{e.axeq})
and using the variables $\Gamma:=\sqrt{1-\gamma^2}/\gamma$ $ \in (0,\infty]$,
$\Delta=a^2$, $b^2=\Delta(1-\kappa)/(1+\kappa\Delta)$ 
with 
$$ \kappa=\frac{\mu_k-\mu_l}{\mu_j-\mu_l}$$
results in a representation of $\tilde{a}^2/a^2$ as a 
function of $t$, $\Delta$, $\Gamma$ and $\kappa$. 
(The limit $\Gamma\to\infty$ needs additional care; however this limit corresponds to
$\gamma=0$. For $\gamma=0$ Thm.~\ref{t.psdmu} is already proved in \cite{NOZ2011}.)
The details are as follows.
With
\begin{align*}
  A&=\sqrt{1+\alpha_0^2+\beta_0^2}\,(\mu_j-\mu)+\Gamma\alpha_0\beta_0(\mu_k-\mu_l),\\
  B&=\sqrt{1+\alpha_0^2+\beta_0^2}\,\alpha_0(\mu-\mu_k)+\Gamma\beta_0(\mu_j-\mu_l),\\
  C&=\sqrt{1+\alpha_0^2+\beta_0^2}\,\beta_0(\mu-\mu_l)+\Gamma\alpha_0(\mu_k-\mu_j)
\end{align*}
it holds that $c_k=A/B$ and $c_l=A/C$. Instead of considering $\tilde{a}^2/a^2$
it is more convenient to estimate its reciprocal from below.
From (\ref{e.axeq}) one gets 
$$ \frac{a^2}{\tilde{a}^2}=\frac{\Delta(1-\kappa)}{1+\kappa\Delta}\left(\frac CA\right)^2
   +\Delta \left(\frac BA\right)^2$$
with
\begin{align*}
   \frac CA = \frac{ \sqrt{1+\alpha_0^2+\beta_0^2}\,\beta_0+\Gamma\alpha_0
         \frac{\mu_k-\mu_j}{\mu-\mu_l}}
	 { \sqrt{1+\alpha_0^2+\beta_0^2}\,b^2+\Gamma\alpha_0\beta_0
         \frac{\mu_k-\mu_l}{\mu-\mu_l} },\quad
   \frac BA = \frac{ \sqrt{1+\alpha_0^2+\beta_0^2}\,\alpha_0+\Gamma\beta_0
         \frac{\mu_j-\mu_l}{\mu-\mu_k}}
	 { \sqrt{1+\alpha_0^2+\beta_0^2}\,a^2+\Gamma\alpha_0\beta_0
         \frac{\mu_k-\mu_l}{\mu-\mu_k}}.
\end{align*}
In these formula the ratios of eigenvalue differences are to be expressed in terms
of $\Delta$ and $\kappa$. Therefore let $U:=\mu_j-\mu$, $V:=\mu-\mu_k$ and
$W:=\mu-\mu_l$ so that $\mu_k-\mu_l=W-V$, $\mu_j-\mu_l=U+W$ and $\mu_k-\mu_j=-U-V$.
Since $\Delta=U/V$ and $\Delta(1-\kappa)/(1+\kappa\Delta)=U/W$ we get that
\begin{align*}
   \frac{\mu_k-\mu_j}{\mu-\mu_l}&=-\frac UW(1+\frac VU)=\frac{(\kappa-1)(1+\Delta)}
      {1+\kappa\Delta},\\
   \frac{\mu_k-\mu_l}{\mu-\mu_l}&=1-\frac VU \, \frac UW= \frac{\kappa(1+\Delta)}
       {1+\kappa\Delta},\\
   \frac{\mu_j-\mu_l}{\mu-\mu_k}&=\frac{U+W}V=\frac UV(1+\frac WU)=
          \frac{1+\Delta}{1-\kappa},\\
   \frac{\mu_k-\mu_l}{\mu-\mu_k}&=\frac{W-V}{V}=\frac WU\frac UV-1=
        \frac{\kappa(1+\Delta)}{1-\kappa}.
\end{align*}
Therefore we have
\begin{align*}
  \frac{a^2}{\tilde{a}^2}=&\frac{\Delta(1-\kappa)}{1+\kappa\Delta}
   \left(\frac{\sqrt{1+\alpha_0^2+\beta_0^2}\,\beta_0+\Gamma\alpha_0\frac{(\kappa-1)(1+\Delta)}
      {1+\kappa\Delta}}
     {\sqrt{1+\alpha_0^2+\beta_0^2}\,\frac{\Delta(1-\kappa)}{1+\kappa\Delta}
     +\Gamma\alpha_0\beta_0 \frac{\kappa(1+\Delta)}
       {1+\kappa\Delta}} \right)^2 \\
       &+\Delta
       \left(\frac{\sqrt{1+\alpha_0^2+\beta_0^2}\,\alpha_0+\Gamma\beta_0\frac{(1+\Delta)}
      {1-\kappa}}
     {\sqrt{1+\alpha_0^2+\beta_0^2}\,\Delta
     +\Gamma\alpha_0\beta_0 \frac{\kappa(1+\Delta)}
       {1-\kappa}} \right)^2.
\end{align*}
Insertion of (\ref{e.al0be0}) yields $f:=f(\Delta,t,\kappa,\Gamma)$ with
\begin{align*}
f= \frac{a^2}{\tilde{a}^2}=&\Big(
  (1+\Delta)(\Gamma^2(1-\kappa)^2+\kappa(1-\kappa)+\Gamma^2t^2)
  +(1-\kappa)^2+t^2(1-\kappa) \\
  & +2\kappa\Gamma t\sqrt{1/(1+t^2)}\sqrt{1+t^2+\kappa\Delta}\sqrt{1-\kappa}
     \sqrt{1+\Delta}\Big) / \\
   & \left(\sqrt{1-\kappa}\sqrt{1+t^2+\kappa\Delta}
       +\kappa\Gamma t\sqrt{1/(1+t^2)}\sqrt{1+\Delta}\right)^2.
\end{align*}
This function is monotone increasing in $\Delta$ since $\partial f/\partial\Delta$ equals
\begin{align*}
   \frac{\Gamma^2\sqrt{1-\kappa}\Big( (1-\kappa)^3 +3(1-\kappa)^2t^2+3(1-\kappa)t^4+t^6\Big)}
   {(1+t^2)\sqrt{1+t^2+\kappa\Delta}\Big( 
     \sqrt{1-\kappa}\sqrt{1+t^2+\kappa\Delta}+\kappa\Gamma t\sqrt{1/(1+t^2)}
           \sqrt{1+\Delta}
        \Big)^3}>0.
\end{align*}
Therefore $f(0,t,\kappa,\Gamma)$ is a lower bound for $a^2/\tilde{a}^2$ which reads
\begin{align*}
  f(0,t,\kappa,\Gamma)=\frac{(1+t^2)\Big( 
   \Gamma^2(1-\kappa)^2+(1+t^2)(1-\kappa)+\Gamma^2t^2+2\kappa\Gamma t\sqrt{1-\kappa}  
       \Big)}
     {(\sqrt{1-\kappa}(1+t^2)+\kappa\Gamma t)^2}.
\end{align*}
The parameter $t$ determines the choice of $x$ in the level set $\lmu$. The
derivative with respect to $t$ reads
\begin{align*}
  \frac{\partial }{\partial t}f(0,t,\kappa,\Gamma)=
  \frac{ 2\kappa\Gamma^2(1-\kappa+t^2)\Big(\Gamma t^2+2t\sqrt{1-\kappa}-\Gamma(1-\kappa)\Big)}
       {\Big( \sqrt{1-\kappa}(1+t^2)+\kappa\Gamma t\Big)^3}.
\end{align*}
The two real zeros of this derivative are
$$  t_{1,2}=\frac{\sqrt{1-\kappa}(-1\pm\sqrt{1+\Gamma^2})}{\Gamma}.$$
The global minimum is taken in
$$ 0<t_1=\frac{\sqrt{1-\kappa}(-1+\sqrt{1+\Gamma^2})}{\Gamma}
        =\frac{\sqrt{1-\kappa}(1-\gamma)}{\sqrt{1-\gamma^2}}.$$
Therefore the minimum is given by
\begin{align*}
   f(0,t_1,\kappa,\Gamma)=\left( \frac{(2-\kappa)+\gamma\kappa}
           {\kappa+\gamma(2-\kappa)} \right)^2
\end{align*}        
and its inverse yields the desired convergence estimate
$$  \frac{\Delta(\mu')}{\Delta(\mu)}\leq \left( \frac{\tilde{a}}{a}\right)^2\leq 
   \left(\frac {\kappa+\gamma(2-\kappa)}{(2-\kappa)+\gamma\kappa}\right)^2.$$
This estimate is sharp since for $\Delta=0$ the right inequality turns into an 
identity. Further $\Delta=0$ implies $\mu(x)\to\mu_j$ and also $\mu(x')\to\mu_j$ so that
$\lim_{\mu(x)\to\mu_j}\tilde{a}/\tilde{b}-a'/b'=0$ and in this limit
$\mathcal{L}(\mu')\cap{\mathcal E}_j$ and $E$ coincide; this implies
that the left inequality also turns into an identity. 
\end{proof}

\noindent
\begin{proof}[of Theorem \ref{t.psdmu} and Theorem \ref{t.psd}]
Let $\mu=\mu(x)\in (\mu_{i+1},\mu_i)$. Theorem \ref{t.3Djust} proves that the
poorest convergence is attained in a three-dimensional invariant subspace. 
Theorem \ref{t.3Dest} proves in $\spanl\{e_j,e_k,e_l\}$ that 
$$    \frac{\Delta_{j,k}(\mu')}{\Delta_{j,k}(\mu)}\leq \left(
      \frac {\kappa+\gamma(2-\kappa)}{(2-\kappa)+\gamma\kappa}\right)^2.$$
It either holds that $\mu_l\leq \mu_{i+1}\leq\mu(x)<\mu_i \leq\mu_k<\mu_j$ or
that $\mu_l<\mu_k\leq\mu_{i+1}<\mu(x)<\mu_i\leq\mu_j$. In the first case 
the Ritz value $\mu(x')$ in $\spanl\{e_j,e_k,e_l\}$ satisfies that $\mu_k\leq\mu(x')$,
which is the first alternative in Thm.~\ref{t.psdmu}. To analyze the second case we get that
the convergence factor is a monotone increasing function in $\kappa\in(0,1)$ since
$$ \frac{\partial }{\partial \kappa}\frac {\kappa+\gamma(2-\kappa)}{(2-\kappa)+\gamma\kappa}
   =\frac{2(1-\gamma^2)}{(2-\kappa)+\gamma\kappa}\geq 0.$$
Further $\kappa=(\mu_k-\mu_l)/(\mu_j-\mu_l)$ is a monotone decreasing function 
in $\mu_j$ and $\mu_l$ and a
monotone increasing function in $\mu_k$. Hence the poorest convergence with the maximal
convergence factor is attained in $j=i$, $k=i+1$ and $l=n$ which proves Thm.~\ref{t.psdmu}
$$   \frac{\Delta_{i,i+1}(\mu')}{\Delta_{i,i+1}(\mu)}\leq \left(
      \frac {\kappa+\gamma(2-\kappa)}{(2-\kappa)+\gamma\kappa}\right)^2 \quad
      \text{with } \kappa=\frac{\mu_{i+1}-\mu_n}{\mu_i-\mu_n}.   $$
Thm.~\ref{t.psd} follows by inserting the reciprocals of the eigenvalues and Ritz values.
 \end{proof}

\section*{Conclusions}
The new convergence bound given in Theorem \ref{t.psd} completes the efforts 
to find sharp convergence estimates within the hierarchy of preconditioned PINVIT($k$) and
non-preconditioned INVIT($k$) eigensolvers for the index $k=2$; a hierarchy 
of these solvers has been suggested in \cite{NEY2001H}.
Next the results are summarized. All these convergence estimates have the common
form
$$   \Delta_{i,i+1}(\rho(x')) \leq \sigma^2  \Delta_{i,i+1}(\rho(x)) $$
with $ \Delta_{i,i+1}(\xi)=(\xi-\lambda_i)/(\lambda_{i+1}-\xi)$.

The convergence factor for the non-preconditioned inverse iteration 
 INVIT(1) procedure is  (see \cite{NEY2005})
\begin{align*} 
  \sigma(\text{INVIT(1)})&=\frac{\lambda_i}{\lambda_{i+1}}.\\
\intertext{The associated preconditioned scheme, i.e. the 
  preconditioned inverse iteration PINVIT(1) or
  preconditioned gradient iteration, has the convergence factor  (see \cite{KNN2003})}
  \sigma(\text{PINVIT(1)})&=\gamma+(1-\gamma)\frac{\lambda_i}{\lambda_{i+1}}.
\intertext{Further the convergence factor of the non-preconditioned steepest descent iteration 
INVIT(2) reads (see \cite{NOZ2011})}
 \sigma(\text{INVIT(2)})&=\frac{\kappa}{2-\kappa}\quad \text{ with }
      \kappa= \frac{\lambda_i(\lambda_n-\lambda_{i+1})}{\lambda_{i+1}(\lambda_n-\lambda_i)}.  \\
\intertext{The new result on PINVIT(2), which is the preconditioned steepest descent iteration, is now}
  \sigma(\text{PINVIT(2)})&=\frac{\kappa+\gamma(2-\kappa)}{(2-\kappa)+\gamma\kappa} 
  \quad \text{ with }
      \kappa=\frac{\lambda_i(\lambda_n-\lambda_{i+1})}{\lambda_{i+1}(\lambda_n-\lambda_i)}.
\end{align*}
All these convergence factors are sharp.

Further progress in deriving convergence estimates for the hierarchy of non-pre\-con\-di\-tioned
and preconditioned iteration is a matter of future work. Especially for the practically important
locally optimal preconditioned conjugate gradient (LOPCG) iteration \cite{KNY1998} sharp
convergence estimates are highly desired.

\section{Acknowledgment}
The author is very grateful to Ming Zhou, University of Rostock, for his help with
the introduction of the ellipse $E$ in Section \ref{ss.levset},
which was a valuable input to finalize the convergence proof. 

\bibliographystyle{siam}

\def\cprime{$'$} \def\cprime{$'$}

\end{document}